\documentclass[twoside,12pt]{amsart}

\usepackage{amssymb,amsfonts,amsmath}

\usepackage[colorlinks]{hyperref}

\newtheorem{theorem}{Theorem}

\newtheorem{corollary}[theorem]{Corollary}

\theoremstyle{remark}
\newtheorem{definition}[theorem]{Definition}
\newtheorem{remark}[theorem]{Remark}
\newtheorem{proposition}[theorem]{Proposition}
\newtheorem{example}[theorem]{Example}

\newtheorem*{acknowledgement}{Acknowledgment}
\numberwithin{theorem}{section}

\numberwithin{equation}{section}

\def\part{\partial}
\newcommand{\R}{{\mathord{\mathbb R}}}

\begin{document}

\title[Global Solutions to Bubble Growth in Porous Media]
{Global Solutions to Bubble Growth in Porous Media}

\author[L. Karp]{Lavi Karp}
\email{karp@braude.ac.il}
\address{Department of Mathematics,
         ORT Braude College,
         P.O. Box 78,
         21982 Karmiel,
         Israel}\thanks{*Supported by DFG Ref: SCHU 808/21-1 and by ORT
Braude College's Research Authority}


\begin{abstract}
We study a moving boundary problem modeling an injected fluid into another
viscous fluid.  The viscous fluid is withdrawn at infinity and governed by
Darcy's law. We present solutions to the free boundary problem in terms of
time-derivative of a  generalized Newtonian potentials of the characteristic
function of the bubble. This enables us to show that
the bubble occupies the
entire space as the time tends to infinity if and only if the internal
generalized Newtonian potential of the initial bubble is a quadratic polynomial.
Howison \cite{Ho}, and DiBenedetto and Friedman \cite{BF}, studied such
behavior, but for bounded bubbles. We extend their results to unbounded
bubbles.
\end{abstract}

\maketitle

\section{Introduction}

The propose of this note is to apply the technique of generalized Newtonian
potential to the study of  bubble growth in porous media and Hele-Shaw flows.
This technique  enables the computation  a potential of measures with unbounded
support in a similar manner to the ordinary Newtonian potential . It is a
multi-valued right inverse to the Laplacain and unique up to a harmonic
polynomial of degree not exceeding two (see \cite{Ka1, KM, Ma1}). We can
therefore compute the potential of the characteristic function of arbitrary
large set $\R^n$. The Newtonian potential theory is a basic tool  in the 
studies of  Hele-Shaw flows \cite{EE, G, Ma3,  Sa2, VE}, and   since in a bubble
  growth in porous media the fluid motion is in unbounded domains, the
application of the generalized Newtonian potential is rather natural in this
type of moving boundary problems.

Consider $\R^n$  as homogeneous porous medium filled with a
viscous fluid. Another fluid is injected and forms a bubble which
occupies a domain  $D(t)$ at each time $t$. The fluid withdrawn at
infinity at a certain rate and the bubble $D(t)$ increases with the
time $t$. Let $\Omega(t)=\mathbb{R}^n\setminus D(t)$ and $\Psi(x,t)$ denotes the
pressure of the incompressible fluid in
$\Omega(t)$.  We will prove that $\Psi(x,t)$,  the solution to this moving
boundary problem (\ref{eq:formulation}) below,
is a time-derivative of a generalized Newtonian
potential of $\chi_{D(t)}$,  the characteristic function of bubble $D(t)$. In
this formulation it does  not matter whether the bubble is bounded or unbounded.

The representation of solutions by potentials has several advantages specially
in higher dimensions where the tool of conformal mappings is not available. For
example,  this enables us to construct in a simple manner solutions  for which
the  bubble  exists for all for all $t>0$ and occupies the entire space as as
$t\to\infty$.  This problem was settled  by Howison \cite{Ho}, and DiBenedetto
and  Friedman \cite{BF}, but under the condition that the initial bubble is a
bonded domain. Here we show that this type of fluids motions exists if and only
if the internal generalized Newtonian  potential of the initial bubble
$\chi_{D(0)}$ is a quadratic polynomial. It thus  extends their result to
unbounded bubbles.

The known examples of domains in $\mathbb{R}^n$ for which  internal generalized
Newtonian  potential of their characteristic set  equals to a quadratic
polynomial are  (a) ellipsoids, (b) convex domains bounded by elliptic
paraboloids,  (c) domains bounded by two parallel hyperplanes, (d) cylinders
over (a) and (b), and (e) half-spaces \cite{KM}. The complements of these
domains are null quadrature domains, that is, an open set $\Omega$ for which
\begin{equation*}
\int_{\Omega}hdx=0
  \end{equation*}
for all harmonic and integrable functions $h$ in $\Omega$.   In the two
dimensional plane  Sakai
classified those domains \cite{Sa1}, but   the classification in higher
dimensional spaces is an open problem.  Dive \cite{D}, and Nikliborc \cite{Ni},
proved that if the internal Newtonian potential of a bounded domain is a
quadratic polynomial, then it must be an ellipsoid (their proof in given in
$\mathbb{R}^3$, for proofs in arbitrary dimension see \cite{BF, Ka2}). The
classification  problem of this type of domains is also settled under the a
priori assumption that the
domain is contained in a cylinder of
co-dimension two \cite{KM}.  A recent progress in the classification problem
was obtained in \cite{KM2} and for its current state  see Remark \ref{rem}
below. 

\subsection{The  formulation of the moving boundary problem}

The bubble  is formed by injected a fluid of negligible viscosity into another
incompressible viscous fluid. We denote the set occupied by the bubble at time
$t$  by  $D(t)$ and $\Omega(t)=\mathbb{R}^n\setminus D(t)$. Suppose  $\R^n$ to
consists of homogeneous porous medium,  then the motion of the flow in
$\Omega(t)$ is subject to the Darcy's law which asserts that the velocity
$\vec{v}$ is proportional to the gradient of the pressure $\Psi$:

\begin{equation*}
\vec{v}=-\kappa\nabla_x\Psi.
\end{equation*}
Since the flow is incompressible, ${\rm div}\ \vec{v}=0$, hence
\begin{equation*}
\Delta \Psi =0 \quad \text{in} \quad \Omega(t)\quad \text{for all}
\quad t>0.
\end{equation*}
Assuming there is no surface tension, then the pressure is constant on the
boundary $\part\Omega(t)$, so we
may set
\begin{equation*}
\Psi(x,t) =0 \quad x\in D(t)\quad \text{for all} \quad t>0.
\end{equation*}
The free boundary  moves with velocity $-\nabla_x \Psi$, that
is
\begin{equation*}
\frac {\part \Psi}{\part\eta}=-\mathcal{V}_\eta  \quad  \text{on} \quad
\part D(t),
\end{equation*}
where $\eta$ is the outward normal pointing into $\Omega(t)$ and
$\mathcal{V}_\eta$ is is the outward velocity of the free boundary $\part
D(t)$. If $D(t)=\{g(x,t)<0\}$, then $\frac
{\part\Psi}{\part\eta}=\nabla_x\Psi\cdot\left(\frac {\nabla_x
g}{|\nabla_x g|}\right)$ and $\mathcal{V}_\eta=\frac {-\part_t g}{|\nabla_x
g|}$. Generalized Newtonian potential of densities in
$L^\infty(\mathbb{R}^n)$ have $|x|^2\log|x|$ growth as $x\to \infty$
\cite{Ka1,Shp1}, and 
since in our setting the bubble $D(t)$ may has infinite volume we allow
$\Psi(x,t)=o(|x|^3)$ at infinity.
We summarize these conditions:

\begin{subequations}\label{eq:formulation}
  \begin{eqnarray}
&& \qquad
        \Delta \Psi=0 \quad  \text{in} \quad\Omega(t),
        \label{f.a}\\
&& \qquad \Psi=0 \quad \text{in}\quad D(t),
        \label{f.b}\\
&& \qquad \frac{\part\Psi}{\part\eta}=-\mathcal{V}_\eta \quad
\text{on}\quad
\part \Omega(t),
        \label{f.c}
          \\
&& \qquad \Psi(x,t)=o(|x|^{3})\quad  \text{as}\quad
   |x|\to\infty,\label{f.d}
   \\
&& \qquad  \nabla\Psi(x,t)=o(|x|^{2})\quad \text{as}\quad
   |x|\to\infty.
        \label{f.e}
  \end{eqnarray}
\end{subequations}

\section{Generalized Newtonian Potential}

The   definition   and establishment of basic  properties of this potential were
carried out in \cite{KM}. Here we recall the definition and  present  an
estimate which is needed for the current application.

The Newtonian potential of a measure $\mu$ with compact support is
defined by means of the convolution
\begin{equation}
\label{eq:27}
    V(\mu)(x)=\left( J\ast\mu \right)(x)=\int J(x-y)d\mu(y)
\end{equation}
where
\[
    J(x)=\left\{\begin{array}
        {ll}-\dfrac{1}{2\pi}\log |x| ~~, \quad &n=2\\
        \dfrac{1}{(n-2)\omega_n |x|^{n-2}} ~~, \quad &n\geq 3
    \end{array}\right.
\]
and $\omega_n$ is the area of the unit sphere in $\R^n$. The potential
$V(\mu)$ satisfies the Poisson equation $\Delta V(\mu)=-\mu$ in the
distributional sense. Similarly, the {\it generalized Newtonian
potential} is a right inverse of the Laplacian, but it is  multi-valued and
acting on the space  $\mathcal{L}$, the space of all Radon measures $\mu$ in
$\R^n$ satisfying condition
\begin{equation}
\label{eq:11}
\left\|\mu\right\|_{\mathcal{L}} :=
\int \frac{d|\mu|(x)}{1+|x|^{n+1}} < \infty.
\end{equation}
Note that $\mathcal{L}$ contains all measures with densities in
$L^\infty(\R^n)$. The linear space $\mathcal{L}$ is the Banach space with the
norm
defined by (\ref{eq:11}).
For $\mu \in \mathcal{L}$ we first define $V^\alpha(\mu)$,  the operator of
the third order generalized derivatives of the potential:
\begin{equation}
 \label{eq:2}
\langle V^\alpha(\mu), \varphi \rangle := - \int \part^\alpha V(\varphi)(x)
d\mu(x),
    \quad \varphi\in \mathcal{S},\ |\alpha|=3.
\end{equation}
Here $\mathcal{S}$ is the Schwartz class of rapidly decreasing functions and
$V^\alpha(\mu)$ is a tempered distribution.

\begin{definition}
\label{Generalized Newtonian potential}
The generalized Newtonian potential $V[\mu]$ of a measure $\mu\in \mathcal{L}$
is the set of
all solutions to the system
\begin{equation}
\label{eq:28}
    \left\{
    \begin{array}
        {l}\Delta u=-\mu \\
        \part^\alpha u=V^\alpha(\mu) ~,\quad |\alpha|=3~~.
    \end{array} \right.
\end{equation}
\end{definition}
The existence of solutions to system (\ref{eq:28}) was proved in
\cite{KM} and it is unique modulo
$\mathcal{H}_2
$, the space of all harmonic polynomials of degree at most two. The
operator $V \colon \mathcal{L} \to \mathcal{S}^\prime / \mathcal{H}_2$
is continuous \cite{KM}.

We will not distinguish here between ordinary and generalized potentials. By a
potential of a set $A$ we mean the potential of its characteristic function
$\chi_A$. 

Let  $k$ be a positive integer, then
\begin{equation*}
 \|\varphi\|_k=\sup_{\{|\alpha|\leq
k\}}\sup_{\mathbb{R}^n}\left(\left(1+|x|\right)^k|\partial^k \varphi(x)|\right)
\end{equation*}
is a semi-norm on $\mathcal{S}$.  We shall need the following estimate. 

\begin{proposition}
\label{prop:5} Let $\alpha$ and $\beta$ be multi-indexes such that
$|\alpha|=3$. Then for any $\varphi\in \mathcal{S}$
\begin{equation}
\label{eq:1}
|\part^{\alpha+\beta}V(\varphi)(x)\leq C
\left(1+|x|\right)^{-(n+1+|\beta|)}\|\varphi\|_{2\left(n+1+|\beta|\right)},
\end{equation}
where the constant $C$ does not depend on $\varphi$.
\end{proposition}

For $\beta=0$ this was proved in \cite[\S 1]{KM}. Since only a slightly
modification is needed in order to extend it for (\ref{eq:1}) we leave it to
the reader.

\section{The Generalized Newtonian  Potential of the Bubble }

In this section we show that solutions to the moving boundary problem
(\ref{eq:formulation}) can be expressed in terms of the generalized potential.

\begin{proposition}
\label{prop:1}
Let  $V^\alpha(\chi_{D(t)})$ be the third order derivatives operator defined in
(\ref{eq:2})
and assume $\Psi$ satisfies (\ref{eq:formulation}), then
\begin{equation}
\label{eq:5} \frac{d}{dt}V^\alpha(\chi_{D(t)})=\part^\alpha \Psi,\quad
|\alpha|=3,
\end{equation}
where the identity (\ref{eq:5}) is in the distributional sense.
\end{proposition}

\begin{proof}
Let $\varphi\in \mathcal{S}$,  then Proposition \ref{prop:5} implies that $\part
V^\alpha(\varphi)\in L^1(\R^n)$ and  therefore
\begin{math}
 V^\alpha(\chi_{\mathbb{R}^n})(\varphi)=-\int_{\mathbb{R}^n}\part^\alpha
V(\varphi)dx=0.
\end{math}
Since $\Omega(t)=\mathbb{R}^n\setminus D(t)$,
\begin{equation}
\label{eq:6} \langle
V^\alpha(\chi_{D(t)}),\varphi\rangle=\int_{D(t)}\part^\alpha
V(\varphi)(x)dx=-\int_{\Omega(t)}\part^\alpha V(\varphi)(x)dx
\end{equation}
and
\begin{equation}
\label{eq:7} \frac{d}{dt}\left(\langle
V^\alpha(\chi_{D(t)}),\varphi\rangle\right)=-\int_{\part\Omega(t)}\part^\alpha
V(\varphi)(x)\mathcal{V}_\eta dS.
\end{equation}
Let $B_R$ be a ball of radius $R$ and apply the Green's formula, then
\begin{equation}
\label{eq:8}
\begin{split}& \int_{\Omega(t)\cap B_R}\Delta
\left(\part^\alpha V(\varphi)\right)\Psi
dx \\ =-&\int_{\part\Omega(t)}\left(\frac{\part \left(\part^\alpha
V(\varphi)\right)}{\part \eta}\Psi-\part^\alpha
V(\varphi)\frac{\part \Psi}{\part \eta}\right)dS\\
+&\int_{\Omega(t)\cap \part B_R}\left(\frac{\part
\left(\part^\alpha V(\varphi)\right)}{\part \eta}\Psi-\part^\alpha
V(\varphi)\frac{\part \Psi}{\part \eta}\right)dS.
\end{split}
\end{equation}
By the estimate (\ref{eq:1}) and the growth assumptions (\ref{f.d}) and
(\ref{f.e}) of $\Psi$,
the second term of the right hand site of (\ref{eq:8}) tends to
zero as $R$ goes to infinity. Thus, letting $R\to \infty$ and using the
boundary conditions of (\ref{eq:formulation}) we get
\begin{equation}
\label{eq:21}
 \int_{\Omega(t)}\Delta
\left(\part^\alpha V(\varphi)\right)\Psi
dx =\int_{\part\Omega(t)}\part^\alpha V(\varphi)\mathcal{V}_\eta dS.
\end{equation}
Since  $\Delta V(\varphi)= -\varphi$  for any test function, we see from  
(\ref{eq:7}) and (\ref{eq:21}) that
\begin{equation}
\label{eq:9} \frac{d}{dt}\left(\langle
V^\alpha(\chi_{D(t)}),\varphi\rangle\right)=-\int_{\Omega(t)}
\Delta\left(\part^\alpha
V(\varphi)\right)\Psi dx=-\int\part^\alpha \varphi\Psi dx
\end{equation}
which is the distributional meaning of (\ref{eq:7}).
\end{proof}

\begin{corollary}
\label{cor:1}
If the bubble $D(t)$ occupies the entire space as $t$ tends to infinity,
that is,
\begin{equation}
\label{eq:12} \lim_{t\to\infty}D(t)=\R^n,
\end{equation}
then the potential of the initial bubble $D(0)$ coincides with a
quadratic polynomial on  $D(0)$.
\end{corollary}

\begin{proof}
Taking the integral of  (\ref{eq:5}) we get 
\begin{equation}
\label{eq:10}
\langle V^\alpha(\chi_{D(t)}),\varphi\rangle-\langle
V^\alpha(\chi_{D(0)}),\varphi\rangle\\
 =-\int_0^t\int\part^\alpha \varphi(x)\Psi(x,s) dxds.
\end{equation}
Since
$\Psi(x,s)=0$ for all $x\in D(0)$ and all $s\geq 0$, we get from (\ref{eq:10})
that
\begin{equation}
 \langle V^\alpha(\chi_{D(0)}),\varphi\rangle=\langle
V^\alpha(\chi_{D(t)}),\varphi\rangle \quad \text{for all}\ \varphi\in
C_0^\infty(D(0)).
\end{equation}
Now if (\ref{eq:12}) holds, then by the continuity of the generalized
potentials we have that
\begin{equation*}
  \langle V^\alpha(\chi_{D(0)}),\varphi\rangle=\lim_{t\to \infty}\langle
V^\alpha(\chi_{D(t)}),\varphi\rangle=\langle
V^\alpha(\chi_{\mathbb{R}^n}),\varphi\rangle=0
\end{equation*}
for all $\varphi\in C_0^\infty(D(0))$.
Hence  the third order derivatives for any $u\in V[\chi_{D(0)}]$
vanish in $D(0)$. 
\end{proof}

The following formula is the analogous of Richardson's theorem \cite{Ri}, but
for suction at infinity. This was previously proved by Entov and Etingof
\cite{EE} (see also  \cite[\S 4.6]{VE}), in the plane and for a bounded bubble.
We thus
generalized
their result to the space $\R^n$ and for unbounded bubbles.

\begin{theorem}
\label{thm:1}
Suppose the solution of (\ref{eq:formulation})
exists for $t\in [0,T])$, then there exists a generalized Newtonian potential 
$u(\cdot, t)\in
V[\chi_{D(t)}]$ such that
\begin{equation}
\label{eq:13} \frac{d }{dt}u(x,t)=a(t)+\Psi(x,t).
\end{equation}
\end{theorem}

\begin{proof}
Let $w\in V[\chi_{D(t)}]$ and $|\alpha|=3$, then by Proposition \ref{prop:1}
\begin{equation}
\frac{d}{dt}\part^\alpha w(x,t)=\part^\alpha\Psi(x,t),
\end{equation}
which implies
\begin{equation}
\part^\alpha w(x,t)=F_\alpha(x)+\int_0^t\part^\alpha\Psi(x,s)ds.
\end{equation}
Since $\part_i F_\alpha=\part_j F_\beta$ for any $\alpha,\beta$
with $\part_i \part^\alpha=\part_j \part^\beta$, there is a
function $F$ such that $\part^\alpha F=F_\alpha$. Hence
\begin{equation}
 w(x,t)=F(x)+\int_0^t\Psi(x,s)ds + q(x,t),
\end{equation}
where $q(x, t)$ is a quadratic polynomial. We can write it in the form
\begin{equation*}
q(x,t)=|x|^2M(t)+h(x,t)+A(t),
\end{equation*}
where  $h(x ,t)$ is harmonic polynomial of degree not exceeding  two (see
e.g. \cite{ABR}). Now $\Psi(x,t)=0$ for $x\in D(0)$ and $t\in[0,T]$, therefore
in
$D(0)$ we have
\begin{equation}
-1=\Delta w(x,t)=\Delta F(x)+ 2n M(t).
\end{equation}
Differentiating  with respect to $t$ gives
\begin{equation}
0=\frac{d}{dt}\Delta w(x,t)= 2n \frac{d}{dt}M(t).
\end{equation}
So $M(t)$ does not depend on $t$. Letting $u(x,t)=w(x,t)-h(x,t)$,
then $u$ belongs to $ V[\chi_{D(t)}]$ and satisfies (\ref{eq:13})
 with $a(t)=\frac{d}{dt}A(t)$.
\end{proof}

We shall now see that the converse statement to Theorem \ref{thm:1} is also
true.

\begin{theorem}
\label{thm:2}
 Let $D(t)$ be a continuous family of domains such that $D(t_1)\subset D(t_2)$
for
$t_1<t_2$ and $\part D(t)$ is sufficiently smooth so that the Green's identity
holds. If for $t\in [0,T]$ there are  potentials $u(\cdot,t)\in V[\chi_{D(t)}]$
such that
\begin{equation}
\label{eq:17}
u(x,t_2)- u(x,t_1)\quad \text{does not depend on }
 x   \text{ for }  x\in D(t_1),
\end{equation}
then the solution to (\ref{eq:formulation}) is given by
\begin{equation}
\label{eq:18} \Psi(x,t):=\frac{d}{dt}u(x,t)-a(t).
\end{equation}
\end{theorem}

\begin{proof}
Here we argue similar to \cite{BF}. Let $\varphi\in C_0^\infty(\mathbb{R}^n)$
and $t_1<t_2$. Then
\begin{equation*}
 \lim_{t_2\to t_1}\dfrac{1}{t_2-t_1}\int_{D(t_2)\setminus D(t_1)} \varphi dx=
\int_{\part D(t_1)}\mathcal{V}_\eta dS,
\end{equation*}
where $\mathcal{V}_\eta$ is is the outward velocity of the free boundary $\part
D(t_1)$.  Since
\begin{equation*}
 \dfrac{1}{t_2-t_1}\int_{D(t_2)\setminus D(t_1)} \varphi
dx=\dfrac{-1}{t_2-t_1}\int\left(u(x,t_2)-u(x,t_1)\right) \Delta\varphi dx,
\end{equation*}
 $\frac{d}{dt}u(x,t)$
exists and
\begin{equation}
\label{eq:19}
 \langle \frac{d}{dt}u(\cdot,t),\Delta
\varphi\rangle=-\langle\mathcal{V}_n,\varphi\rangle.
\end{equation}
By condition (\ref{eq:17}), $u(x,t)-u(x,0)$ does not depend on $x$  for $x\in
D(0)$.  So we denote $u(x,t)-u(x,0)$ by $A(t)$ and define
\begin{equation*}
 \Psi(x,t):=\frac{d}{dt}u(x,t)-A'(t).
\end{equation*}
 We shall now verify that this $\Psi(x,t)$
satisfies (\ref{eq:formulation}). Since $u(x,t)\in V[\chi_{D(t)}]$, it is
harmonic in $\Omega(t)$ and therefore (\ref{f.a}) holds. Condition
(\ref{eq:17}) implies that $\Psi(x,t)=0$ for $x\in D(t)$, so (\ref{f.b}) holds
too and the growth properties (\ref{f.d}) and (\ref{f.e}) are consequence of
the known  estimates of the generalized potential of
$L^\infty(\mathbb{R}^n)$-densities \cite{Ka1,Shp1}.  In order to verify
(\ref{f.c}) we use (\ref{f.b}), (\ref{eq:19}) and the Green's formula, these
yield
\begin{equation}
\begin{split}
-\langle\mathcal{V}_\eta,\varphi\rangle &=
 \langle \frac{d}{dt}u(\cdot,t),\Delta\varphi\rangle
=\int_{\Omega(t)}\frac{d}{dt}u(x,t)\Delta\varphi(x) dx\\
& =
\int_{\part D(t)}\left(\frac{\part}{\part\eta}\left(\frac{d}{dt}u(x,
t)\right)\varphi(x)\right)dS.
\end{split}
\end{equation}
This completes the
proof. 
\end{proof}

\section{Construction of global solutions}

In this section we use Theorem \ref{thm:2} in order to construct solutions to
the moving boundary problem (\ref{eq:formulation}) such that the bubble $D(t)$
will occupy the entire space as $t$ goes to infinity.
By Corollary \ref{cor:1} a necessary condition for  that is
\begin{equation*}
 u(x)= q(x) \quad \text{for}\  x \in D(0), \quad \ u\in V[\chi_{D(0)}],
\end{equation*}
 where $q$ is a quadratic polynomial. Here we shall
prove that this is also a sufficient condition.

We first recall few known facts. A classical theorem of Newton says that the
gravitational force of a homogeneous
shell between two similar ellipsoids vanishes in the cavity of the shell (see
e.g. \cite {Ch, Ke}). This is equivalent to the property that the ordinary
internal Newtonian
potential of ellipsoids is a quadratic polynomial (see e.g. \cite{BF}). We also
know that 
the following  unbounded domains have their internal  generalized potential 
coincides with a quadratic polynomial: ({\it i}) strips, ({\it ii}) half-spaces,
 ({\it iii}) convex domains
bounded by elliptical paraboloid, and ({\it iv}) cylinders
over ellipsoids and over these domains
\cite{KM, Shg}.

The first example shows that there are many motions of bubbles which fill the
entire space as $t$ runs to infinity.

\begin{example}
\label{ex:1}
 Let $E(t)$ be any  continuous family of ellipsoids such that $E(t_1)\subset
E(t_2)$ whenever $t_1<t_2$ and $\lim_{t\to\infty}E(t)=\R^n$. Then the Newtonian
potential
$V(\chi_{E(t)})(x)=q(x,t)$ for $x\in
E(t)$, where $q$ is a quadratic polynomial satisfying $\Delta q=-1$. So we may
write
\begin{equation}
\label{eq:15}
 q(x,t)=\sum_{i=1}^n a_i(t)x_i^2 +p_2(x,t), \quad
a_1(t)+\cdots+a_n(t)=-\dfrac 1 2,
\end{equation}
where $p_2$ is a harmonic polynomial of degree not exceeding two.
Then
\begin{equation*}
 h(x,t):=\sum_{i=1}^{n-1}a_i(t)x_i^2+\left(a_n(t)+\dfrac 1 2\right)x_n^2
\end{equation*}
is harmonic polynomial of degree two and therefore
\[u(x,t)=V(\chi_{E(t)})(x)-h(x,t)-p_2(x,t)\]
is a generalized Newton potential of
$\chi_{E(t)}$ satisfying $u(x,t)=-\frac 1 2 x_n^2$
on $E(t)$.  Therefore, condition (\ref{eq:17})
holds so we conclude that $\Psi(x,t)=\frac{d}{dt}u(x,t)$ is a global solution to
the
moving boundary problem (\ref{eq:formulation}).
\end{example}

Of course these solutions may  not  have a meaningful physical interpretation,
since they have  a quadratic growth at infinity. It is therefore reasonable
to look for solutions which have  a minimal  growth at infinity.
We see from (\ref{eq:15}) that a necessary condition for that is that the
coefficients $a_i(t)$ are independent of $t$. Such behavior of the coefficients
$a_i(t)$ occurs 
when $E(t)$ are concentric ellipsoids having there center at the origin (see
e.g. \cite{Ni}). 

Thus the ellipsoids of the form
\begin{equation*}
 E(t)=\left\{x\in\mathbb{R}^n: \left(\frac{x_1}{\alpha_1}\right)^2+\cdots+
\left(\frac{x_n}{\alpha_n}\right)^2<(1+t)^2\right\}
\end{equation*}
will form the  bubble for which $\Psi(x,t)$ will have a minimal growth at
infinity. This example where also obtained by Friedman and DiBenedetto
\cite{BF}.  In addition, any cylinder
\begin{equation*}
 E_k(t)=\left\{x\in\mathbb{R}^k: \left(\frac{x_1}{\alpha_1}\right)^2+\cdots+
\left(\frac{x_k}{\alpha_k}\right)^2<(1+t)^2\right\}\times \mathbb{R}^{n-k}
\end{equation*}
will also gives global solutions which has minimal a growth in directions off
the cylinder.

We now turn  showing that elliptical paraboloid also form a global
solution.

\begin{example}
\label{ex:2}
 Let
\begin{equation}
\label{eq:23}
F=\left\{x:\ x_n>\left(\frac{x_1}{\alpha_1}\right)^2+\cdots
+\left(\frac{x_{n-1}}{\alpha_{n-1}}\right)^2\right\},
\end{equation}
then any  generalized Newton potential of $\chi_F$ coincides with
a quadratic polynomial on $F$. We could therefore argue like as in Example
\ref{ex:1}, but
that will not have a minimal
growth. We claim that there is $u\in V[\chi_F]$ such
that $u(x)$ does not depend on $x_n$ for $x\in F$. Assuming it for the moment,
 and define the bubble $F(t)$ by a translation in the $x_n$ axis, that is,
\begin{equation*}
 F(t)=\{x\in \mathbb{R}^n: (x_1,...,x_{n-1},x_n+t)\in D\}.
\end{equation*}
 Then obviously
$u(x,t):=u(x_1,...,x_{n-1},x_n+t)$ is a generalized potential of $\chi_{F(t)}$
and since $u(x)$ does not depend on $x_n$, $u(x,t)=u(x) $ for $x\in F(t)$.
Thus $u(x,t)$
satisfies condition (\ref{eq:17}) and by Theorem \ref{thm:2}
$\frac{d}{dt}u(x,t)$
is a global  solution to (\ref{eq:formulation}). In addition,  $\lim_{t\to
\infty}F(t)=\mathbb{R}^n$.

In order to prove the claim, we adopt the following argument from
\cite[\S 4]{KM2}.  There we showed that if  $V[\chi_D]$ coincides with a
quadratic polynomial in $D$ and
\begin{equation}
\label{eq:25}
 \lim_{\rho\to\infty}\dfrac{{\rm Vol}(D\cap B_\rho)}{{\rm Vol}(B_\rho)}=0,
\end{equation}
then there is a potential $w\in V[\chi_D]$,
such that
\begin{equation*}
 w(x)=q(x):=\sum_{ij=1}^n a_{ij}x_ix_j +\sum_{i=1}^n b_ix_i +c, \quad
\text{for}\
x\in D
\end{equation*}
 and
\begin{equation}
\label{eq:18}
 \lim_{\rho\to\infty}\dfrac{w(\rho x)-q(\rho x)}{\rho^2}=\sum_{ij=1}^n
a_{ij}x_ix_j.
\end{equation}
Here ${\rm Vol}(\cdot)$ denotes the Lebesgue measure in $\R^n$ and $B_\rho$ a
ball with radius $\rho$.

Now the set $F$ given by (\ref{eq:23}) certainly satisfies condition
(\ref{eq:25}). So let $w\in V[\chi_F]$ be the potential which satisfies
(\ref{eq:18}).  
Since $w(x)-q(x)$ vanishes on $F$, $w(\rho x)-q(\rho x)$ vanishes on the
positive $x_n$ axis for every $\rho$ and therefore the limit in (\ref{eq:18})
also vanishes there. Hence $a_n=0$. Letting   $u(x)=w(x)-\sum_{i=1}^n
(a_{in}+a_ni)x_ix_n -b_nx_n$ gives the required potential which also provides a
solution with  minimal growth at the direction of the negative $x_n$ axis.

\end{example}

\begin{remark}
\label{rem}
The complete characterization of domains $D$ with internal potential equal to a
quadratic polynomial is not known. We recently proved in \cite{KM2} that if $D$
is such domain, then it must be one of the following type:
\begin{enumerate}
 \item[1.] If $\limsup_{\rho\to \infty}\dfrac{{\rm Vol}\left(D\cap
B_\rho\right)}{{\rm Vol}(B_\rho)}>0$, then $D$ is  half-space;
\item[2.]  If $\lim_{\rho\to \infty}\dfrac{{\rm Vol}\left(D\cap
B_\rho\right)}{{\rm Vol}(B_\rho)}=0$, then there are two possibilities:
\begin{enumerate}
 \item[2a.] If $D$ is contained between two parallel hyperplanes, then $D$ is an
ellipsoid, or a cylinder over an ellipsoid;
\item[2b.] If $D$ is not contained between two parallel hyperplanes, then
$D=\{x_n>f(x_1,...,x_k),\ 1\leq k<n  \}$ and $f$ is real analytic convex
function.
\end{enumerate}
\end{enumerate}
\end{remark}

It is obvious that we can extend Example \ref{ex:2} to domains of type 2b.
above. We thus proved:

\begin{theorem}
 Let $\Psi(x,t)$ be a solution to the moving boundary problem
(\ref{eq:formulation}). Then the bubble ${D(t)}$ occupies the entire space as
$t$ goes to infinity if and only if the generalized Newtonian potential of the
initial bubble $V[\chi_{D(0)}]$ coincides with a quadratic polynomial in
$D(0)$. 
\end{theorem}

Finally, these techniques can be used also for contractions of flows. For
example, we can construct solution to (\ref{eq:formulation}) so that the
paraboloid $D$ in (\ref{eq:23}) will contract to a half-line. Here we use again
the fact that there is $u\in V[\chi_D]$ that does not depend on $x_n$ for $x\in
D$.

\begin{acknowledgement}
I would like to thanks A. Margulis for many valuable conversations. 
 This project was completed during the author's visit at the Department
of Mathematics at Potsdam Universit\"at and I am grateful to Professor B.-W. 
Schulze for his support and kind hospitality.
\end{acknowledgement}

\bibliographystyle{amsplain}

\end{document}